\documentclass[12pt]{amsart}
\usepackage{amsmath,amsthm,amsfonts,amssymb,mathrsfs}
\usepackage{inputenc,mathrsfs}

\numberwithin{equation}{section}
\usepackage[dvips]{graphicx}
\usepackage[colorlinks]{hyperref}

\newcommand{\Balpha}{\mbox{$\hspace{0.1em}\rule[0.01em]{0.05em}{0.39em}\hspace{-0.21em}\alpha$}}

\newtheorem{theorem}{Theorem}[section]

\newtheorem{conjecture}{Conjecture}[section]
\newtheorem{question}{Question}[section]
\newtheorem{proposition}{Proposition}[section]
\newtheorem{lemma}{Lemma}[section]
\newtheorem{corollary}{Corollary}[section]

\newcommand{\abs}[1]{\lvert#1\rvert}

\newcommand{\eps}{\varepsilon}

\theoremstyle{remark}

\DeclareMathOperator{\diam}{\mathrm{diam}}

\dedicatory{}


\begin{document}

\title[Intrinsic diameter control under mean curvature flow]{Diameter and curvature control under mean curvature flow}
\author{Panagiotis Gianniotis \and Robert Haslhofer}
\maketitle

\begin{abstract}
We prove that for the mean curvature flow of two-convex hypersurfaces the intrinsic diameter stays uniformly controlled as one approaches the first singular time. We also derive sharp $L^{n-1}$-estimates for the regularity scale of the level set flow with two-convex initial data. Our proof relies on a detailed analysis of cylindrical regions ($\eps$-tubes) under mean curvature flow. The results are new even in the most classical case of mean convex surfaces evolving by mean curvature flow in $\mathbb{R}^3$.
\end{abstract}

\section{Introduction}
A family of hypersurfaces $\{M_t^n\subset\mathbb{R}^{n+1}\}_{t\in [0,T)}$ evolves by mean curvature flow if the normal velocity at each point is given by the mean curvature vector. By classical theory (see e.g. \cite{Huisken84,HuiskenPolden}), given any closed embedded initial hypersurface $M_0^n\subset \mathbb{R}^{n+1}$, there exists a unique smooth solution defined on a maximal time interval $[0,T)$. The maximal existence time $T<\infty$ is characterized by
\begin{equation}
\lim_{t\nearrow T}\max_{M_t}\, \abs{A}=\infty ,
\end{equation}
where $|A|$ denotes the norm of the second fundamental form.

One naturally wonders to what extend one can control the geometry of the hypersurfaces -- curvature integrals, intrinsic diameter, etc -- as one approaches the first singular time (more generally, one can pose these questions also beyond the first singular time in the setting of mean curvature flow with surgery and level set flow, respectively):

\begin{question}[Curvature control]
Can one control the curvature integrals $\int_{M_t} \abs{A}^p \, d\mu$ as the flow approaches the first singular time?
\end{question}

\begin{question}[Diameter control]
Can one control the intrinsic diameter as the flow approaches the first singular time?
\end{question}

The two questions are in fact tightly related. For example, by a result of Topping \cite[Thm. 1.1]{Topping_diameter} we have the estimate
\begin{equation}
\diam (M_t,d_t) \leq C_n \int_{M_t} \abs{H}^{n-1}\, d\mu\, .
\end{equation}

To obtain diameter control, one thus might try to derive uniform $L^{n-1}$-bounds for the mean curvature.

For $n>2$, the situation where there is some hope to get uniform $L^{n-1}$-bounds concerns the evolution of two-convex hypersurfaces,
i.e. when the sum of the smallest two principal curvatures $\lambda_1+\lambda_2$ is positive.\footnote{One of course cannot hope to get $L^{n-1}$-bounds assuming only mean convexity. Indeed, consider the case that $M_0\approx S^{n-2}_r\times S^{2}_R$ is a very thin rotationally symmetric torus, i.e. $r$ is very small. Under the flow these small $(n-2)$-spheres will degenerate to points, and it is easy to see that $\lim_{t\nearrow T}  \int_{M_t} H^{n-1}\, d\mu = \infty$.} This curvature condition arises naturally in the work on mean curvature flow with surgery by Huisken-Sinestrari \cite{HS_surg} (see also \cite{HK_surgery,BH_surgery3d,BH_mod2con}), and its main feature is that it excludes generalized cylinders $\mathbb{R}^j\times S^{n-j}$ with more than one $\mathbb{R}$-factor (i.e. $j\geq 2$) as potential blowup limits.

For $n=2$, the notion of two-convexity boils down to the simpler notion of mean convexity, i.e. positive mean curvature, or in other words the assumption that the flow is moving inwards.

It has been proved by Head \cite{Head} and Cheeger-Haslhofer-Naber \cite{CHN} that for the mean curvature flow of two-convex hypersurfaces one has
\begin{equation}\label{int_bounds}
\int_{M_t} \abs{A}^{n-1-\eps} \, d\mu \leq C(M_0,\eps)<\infty,
\end{equation}
for any $\eps>0$. Motivated by this result, it is natural to conjecture:\footnote{R.H. thanks John Head for introducing him to these conjectures during a visit to the Courant Institute in 2011. While we unfortunately don't know the precise history, the conjectures have certainly been discussed among experts well before 2011, c.f. Perelman's bounded diameter conjecture for 3d Ricci flow \cite[Sec. 13.2]{perelman_entropy}.}

\begin{conjecture}[{$L^{n-1}$-curvature conjecture}]\label{conjecture_curv}
If $\{M_t\subset\mathbb{R}^{n+1}\}_{t\in [0,T)}$ is a mean curvature flow of two-convex closed embedded hypersurfaces, then
\begin{equation}
\int_{M_t} \abs{A}^{n-1} \, d\mu \leq C,
\end{equation}
for some constant $C<\infty$ depending only on the geometric parameters of the initial hypersurface.
\end{conjecture}

\begin{conjecture}[{bounded diameter conjecture}]\label{conjecture_diam}
If $\{M_t\subset\mathbb{R}^{n+1}\}_{t\in [0,T)}$ is a mean curvature flow of two-convex closed embedded hypersurfaces, then 
\begin{equation}
\textrm{diam}({M_t},d_t) \leq C,
\end{equation}
for some constant $C<\infty$ depending only on the geometric parameters of the initial hypersurface.
\end{conjecture}

Note that by the result of Topping \cite[Thm. 1.1]{Topping_diameter} an affirmative answer to Conjecture \ref{conjecture_curv} would imply an affirmative answer to Conjecture \ref{conjecture_diam}.

Conjecture \ref{conjecture_curv} states that one can get rid of the $\eps$ in the curvature estimate \eqref{int_bounds}. The question of whether or not one can actually get rid of the $\eps$ in estimates like \eqref{int_bounds} is a very delicate question, that depends on the fine structure of singularities and regions of high curvature.

For comparison, it is useful to look at related elliptic questions. In recent impressive work \cite{NV}, Naber-Valtorta improved the known $L^{3-\eps}$ estimates for the gradient of minimizing harmonic maps to sharp $L^{3}_{\textrm{weak}}$-estimates. The simple example $f(x)=\tfrac{x}{\abs{x}}$ in dimension three, shows that for minimizing harmonic maps the $L^{3}_{\textrm{weak}}$ estimate actually \emph{cannot} be replaced by an $L^{3}$ estimate; cf. also the related work \cite{NV2,NV3,EE,WangYM}.

Having discussed the subtleties of sharp integral estimates, let us now emphasize that our approach for the solution of Conjecture \ref{conjecture_curv} and Conjecture \ref{conjecture_diam} actually proceeds in the opposite order. Namely, we first prove that we can control the intrinsic diameter:

\begin{theorem}[Intrinsic diameter control]\label{thm_diam}
If $\{M_t\subset\mathbb{R}^{n+1}\}_{t\in [0,T)}$ is a mean curvature flow of two-convex closed embedded hypersurfaces, then
\begin{equation}
\diam (M_t,d_t)\leq C,
\end{equation}
for a constant $C=C(\alpha,\beta,\gamma,\mathcal{A})<\infty$, which only depends on certain geometric parameters of the initial hypersurface $M_0$ (see Section \ref{sec_prelim_mcf}).
\end{theorem}

The diameter bound from Theorem \ref{thm_diam} is new even in the classical case of mean convex surfaces evolving in $\mathbb{R}^3$ (note that for $n=2$ mean convex surfaces are automatically $\beta$-uniformly two-convex with $\beta=1$).

We then use Theorem \ref{thm_diam} to derive sharp integral estimates for the second fundamental form. More precisely, we actually obtain sharp integral estimates for the regularity scale of the level set flow with two-convex initial data. Recall from \cite{CHN}, if $\mathcal{M}=\{M_t\}_{t\in [0,T)}$ is a weak solution of the mean curvature flow (here we only consider the mean convex case, where Brakke solutions and level set solutions are known to be equivalent) then the regularity scale $r_{\mathcal{M}}(x,t)$ at a point $(x,t)\in\mathcal{M}$ is defined as the supremum of $0\leq r\leq 1$ such that $M_{t'}\cap B_r(x)$ is a smooth graph for all $t-r^2<t'<t+r^2$ and such that
\begin{equation}
\sup_{|x'-x|<r, |t'-t|<r^2} r |A|(x',t')\leq 1.
\end{equation}
Obviously $|A|(x,t)\leq r_{\mathcal{M}}(x,t)^{-1}$, but of course a bound for $r_{\mathcal{M}}(x,t)^{-1}$ captures geometric control on a whole parabolic neighborhood of definite size as opposed to just information at the single point $(x,t)$. 

\begin{theorem}[Sharp regularity estimate]\label{thm_regularity_est}
If $\mathcal{M}=\{M_t\subset\mathbb{R}^{n+1}\}_{t\in [0,T)}$ is a level set flow with two-convex initial data, then we have the sharp estimate
\begin{equation}\label{integral_estimate}
\int_{M_t} r_{\mathcal{M}}(x,t)^{-(n-1)} d\mu_t(x)\leq C,
\end{equation}
for a constant $C=C(\alpha,\beta,\gamma,\mathcal{A})<\infty$, which only depends on certain geometric parameters of the initial hypersurface $M_0$ (see Section \ref{sec_prelim_mcf}).
\end{theorem}

Since $|A|(x,t)\leq r_{\mathcal{M}}(x,t)^{-1}$, Theorem \ref{thm_regularity_est} immediately implies a sharp integral estimate for the second fundamental form:
\begin{corollary}[Sharp curvature control]\label{cor_regularity_est}
If $\mathcal{M}=\{M_t\subset\mathbb{R}^{n+1}\}_{t\in [0,T)}$ is a level set flow with two-convex initial data, then we have the sharp curvature estimate\footnote{Note that $M_t$ is $n$-rectifiable with multiplicity one for all $t$, and that the second fundamental form $A$ is defined classically $\mu_t$-almost everywhere.}
\begin{equation}
\int_{M_t} |A|^{n-1} d\mu_t \leq C,
\end{equation}
for a constant $C=C(\alpha,\beta,\gamma,\mathcal{A})<\infty$, which only depends on certain geometric parameters of the initial hypersurface $M_0$ (see Section \ref{sec_prelim_mcf}).
\end{corollary}

In particular, Theorem \ref{thm_regularity_est} and Corollary \ref{cor_regularity_est} crucially improve the regularity estimates from Head \cite{Head} and Cheeger-Haslhofer-Naber \cite{CHN}.

Let us now sketch the main ideas for our proof of Theorem \ref{thm_diam}.

To get some intuition, imagine first the potential scenario of the so-called fractal tube, as illustrated in Figure \ref{fig:fractal_neck}. The worry is that necks might slowly but steadily change their axes as one moves over an uncontrolled number of scales, and arrange themselves in a Koch snowflake like way, to make the intrinsic diameter arbitrarily large.

\begin{figure}
\includegraphics[width=0.9\textwidth]{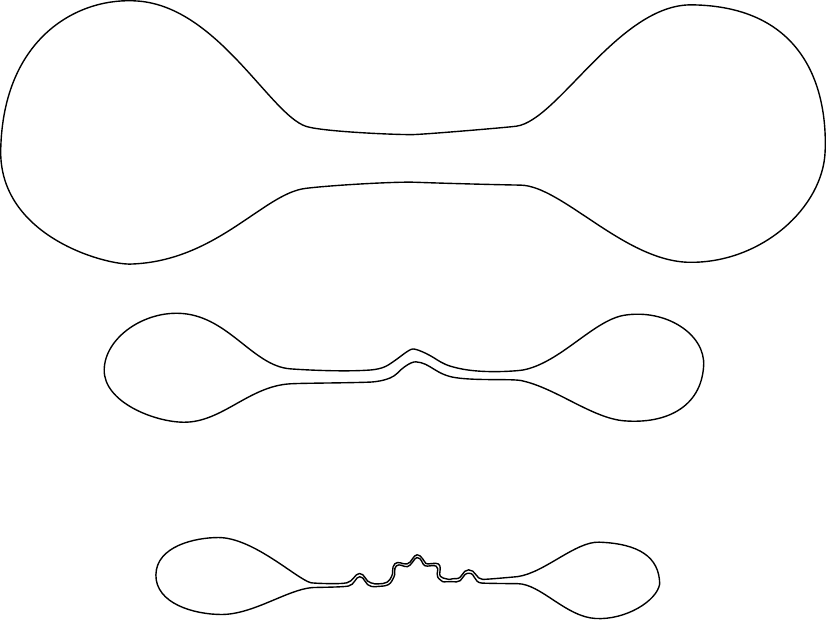}
\caption{Fractal tube}
\label{fig:fractal_neck}
\end{figure}

In a recent breakthrough \cite{CM}, Colding-Minicozzi proved a {\L}ojasiewicz inequality for the mean curvature flow. Their {\L}ojasiewicz inequality implies, roughly speaking, that cylindrical regions can actually only tilt by a controlled amount, \emph{provided} one can ensure a priori that one is very close to a cylinder at all scales in consideration. Colding-Minicozzi applied the {\L}ojasiewicz inequality to prove uniqueness of cylindrical tangent flows \cite{CM}, to derive sharp results about the singular set of mean curvature flow with generic singularities \cite{CM_sing}, and also to derive related sharp results for the arrival time function of a mean convex flow \cite{CM_AT1,CM_AT2}.
In particular, if the flow is two-convex then its space-time singular set is contained in finitely many compact embedded Lipschitz curves together with a countable set of points. This of course immediately rules out the example of an ``exact fractal tube", i.e. a fractal tube as above that becomes singular exactly on a fractal curve.

While the {\L}ojasiewicz inequality from \cite{CM} has been applied many times to derive results about the \emph{singular set}, it hasn't been applied yet in a more quantitative way to derive results about \emph{high curvature regions}. Note that a bound for the size of the singular set by itself, doesn't yield any control for the diameter. For example, one could imagine the scenario of an ``approximate fractal tube" that looks Koch-like shortly before the first singular time, but then only becomes singular with a neck pinching off at one single point.

The main challenge in proving Theorem \ref{thm_diam} is thus to ensure a priori that we can locate enough cylindrical regions propagating all the way from a microscopic scale to a macroscopic scale of definite size, to which we can apply the {\L}ojasiewicz inequality. We achieve this as follows:

Given a two-convex mean curvature flow $\mathcal{M}=\{M_t\subset\mathbb{R}^{n+1}\}_{t\in [0,T)}$, and a time $\bar{t}< T$, we want to derive a bound for $\textrm{diam}(M_{\bar{t}},d_{\bar{t}})$ depending only on certain geometric parameters of the initial hypersurface $M_0$. Using in particular the canonical neighborhood theorem from \cite{HK_surgery}, we first argue (see Proposition \ref{prop_red1} and Proposition \ref{prop_red2}) that it is enough to control the quantity
\begin{multline}
L(M_{\bar{t}}):=\sup\{ \textrm{diam}(N,d_{\bar{t}})\, | \, \textrm{$N\subset M_{\bar{t}}$ is an $\eps$-tube}\\
\textrm{with $H>\bar{H}$ on $N$} \}.
\end{multline}
Here, $\bar{H}$ denotes a very large curvature scale, that in particular is much larger than the curvature scale $H_{\textrm{can}}$ from the canonical neighborhood theorem. An \emph{$\eps$-tube} is a subset $N\subset M_{\bar{t}}$ diffeomorphic to $S^{n-1}\times \mathbb{R}$ such that each $x\in N$ lies on the central sphere of a very strong $\eps$-neck at time $\bar{t}$. Roughly speaking, having a very strong $\eps$-neck means that after rescaling to unit curvature one sees an almost round shrinking cylinder in a very large space-time neighborhood. The precision parameter $\eps$ enters in the rigorous definitions, replacing the informal words ``almost" and ``very large" (see Section \ref{sec_neck_tubes} for the precise definitions). Moreover, we consider a central curve $\gamma$ parametrized by arc-length such that $N$ is covered by $\eps$-necks centered at $p\in \gamma $ and axis parallel to $\partial_s \gamma(p)$.

The central task is to understand how the flow arrived at the tube $N$ at time $t=\bar t$. We show that there is a uniform $\tau>0$ such that around any $p\in \gamma$ and at any time $t\in [\bar t- \tau,\bar t]$ the flow has a strong $\eps_1$-neck (here $\eps \ll \eps_1\ll 1$) centered at $p$ of radius $\sqrt{2(n-1)(t_p-t)}$, where $t_p$ essentially determines the scale of the neck and only depends on $p$. 

To this end, first recall that by the canonical neighborhood theorem all points with $H\geq H_{\textrm{can}}$ have a precise geometric description. The canonical neighborhoods are modeled on so-called $(\alpha,\beta)$-solutions (see Section \ref{sec_can_nbds}), which in turn look neck-like away from cap-like pieces of controlled size. Inspired by \cite{KL_singular}, we prove a backwards-stability result for necks in $(\alpha,\beta)$-solutions (Proposition \ref{prop_backward_stab}). Roughly speaking, the result says that if one sees a neck in an $(\alpha,\beta)$-solution then going far enough back in time one sees an even more precise neck, and in particular not a cap. Combining the neck-stability result with a continuity argument and the assumption that we started off with a very stong $\eps$-tube, we arrive at the conclusion that around any $p\in \gamma$ one sees $\eps_1$-necks of the right radius for a uniform amount $\tau$ backwards in time.

Having done the above, we can finally apply the {\L}ojasiewicz inequality from \cite{CM} to conclude that the total tilt of every neck in $N$, when followed via normal motion from time $\bar{t}$ back to time $\bar{t}-\tau$, is in fact small (Proposition \ref{prop_small_tilt}). Since at time $\bar{t}-\tau$ all necks have macroscopic size, and since overlapping necks must have aligned axes, this gives the desired diameter control.

To prove Theorem \ref{thm_regularity_est}, we first generalize Theorem \ref{thm_diam} to the setting of mean curvature flow with surgery in the framework of \cite{HK_surgery}. The important point is that the diameter bound is independent of the surgery parameters, see \eqref{diam_surgery}. We then combine this with the canonical neighborhood theorem \cite[Thm. 1.22]{HK_surgery} to establish uniform $L^{n-1}$-bounds for the mean curvature under mean curvature flow with surgery, see \eqref{lpmc}.

Given a fixed two-convex initial condition $M_0$, we then consider a sequence $\{M_t^i\}$ of flows with surgery where the surgery parameters degenerate suitably. It follows from the work of Head \cite{Head} and Lauer \cite{Lauer}, that for $i\to \infty$ the sequence of flows with surgery converges to the level set flow. The convergence is in the space-time Hausdorff sense, and also in the sense of varifolds for every time. By lower-semicontinuity, this implies $L^{n-1}$-control for the mean curvature of level set flow with two-convex initial data. Finally, by the local curvature estimate \cite[Thm. 1.8]{HK}, this can be upgraded to $L^{n-1}$-control for the regularity scale.

This article is organized as follows: In Section \ref{sec_prelim}, we introduce our notation and summarize the needed preliminaries. In Section \ref{sec_reduction}, we carry out the reduction of the problem to the case of high curvature tubes. In Section \ref{cylindrical_regions}, we prove the key propositions that have been outlined above and combine them to prove Theorem \ref{thm_diam}. Finally, in Section \ref{sec_sharp_integral}, we generalize the diameter bound to the setting of mean curvature flow with surgery and prove Theorem \ref{thm_regularity_est}.

\textbf{Acknowledgments.} P.G. has been supported by a Fields-Ontario Postdoctoral Fellowship. R.H. has been supported by NSERC grant RGPIN-2016-04331, NSF grant DMS-1406394 and a Connaught New Researcher Award. We thank the Fields Institute and the University of Toronto for providing an excellent research environment. We thank Nick Edelen, John Head, Mohammad Ivaki, Bruce Kleiner, and Felix Schulze for useful discussions.

\section{Notation and preliminaries}\label{sec_prelim}

\subsection{Mean curvature flow}\label{sec_prelim_mcf}
Let $M_0\subset \mathbb{R}^{n+1}$ be a closed embedded hypersurface. Then there exist a unique smooth evolution by mean curvature flow $\mathcal{M}=\{M_t\}_{t\in [0,T)}$ defined on a maximal time interval $[0,T)$. The maximal existence time $T<\infty$ is characterized by 
\begin{equation}
\lim_{t\nearrow T}\max_{M_t}\, \abs{A}=\infty\, .
\end{equation}

A closed embedded hypersurface $M\subset\mathbb{R}^{n+1}$ is called \emph{$2$-convex}, if
\begin{equation}
\lambda_1+\lambda_2>0
\end{equation}
for all $p\in M$, where $\lambda_1\leq \lambda_2\leq \ldots \leq \lambda_n$ are the principal curvatures, i.e. the eigenvalues of the second fundamental form $A$. In particular, every $2$-convex domain has positive mean curvature
\begin{equation}
H>0.
\end{equation}
We also need more quantitative notions of $2$-convexity and embeddedness: A $2$-convex hypersurface $M$ is called \emph{$\beta$-uniformly $2$-convex}, if
\begin{equation}
\lambda_1+\lambda_2\geq \beta H
\end{equation}
for all $p\in M$. An embedded hypersurface $M\subset\mathbb{R}^{n+1}$ is called \emph{$\alpha$-noncollapsed} if it has positive mean curvature and each $p\in M$ admits interior and exterior balls of radius at least $\alpha/H(p)$. By \cite{andrews1,HS_surg}, all the above properties are preserved under the flow.

We also assume that $H\leq\gamma$ initially.  To keep track of the constants, we put $\Balpha=(\alpha,\beta,\gamma)$ and say that $M_0$ is $\Balpha$-controlled. By compactness, every $2$-convex embedded hypersurfaces $M_0$ is $\Balpha$-controlled for some parameters $\alpha,\beta>0,\gamma<\infty$. Some estimates will also depend on a bound $\mathcal{A}$ for the initial area (recall that area is decreasing under the flow). While clearly most constants also depend on the dimension $n$, we usually don't explicitly indicate this in our notation.

\subsection{Necks and tubes}\label{sec_neck_tubes}
In this section, we summarize our basic terminology about necks and tubes.

We say that an embedded hypersurface $M^n\subset\mathbb{R}^{n+1}$ has an \emph{$\eps$-neck} with center $p$ and radius $r$ if the rescaled surface $r^{-1} � (M-p)$ is $\eps$-close in $C^{\lfloor 1/\eps\rfloor}$ in $B_{1/\eps}(0)$ to a round cylinder $S^{n-1}\times\mathbb{R}$ (up to rotation) with center $0$ and radius $1$.

As in \cite[Def. 2.3]{HK_surgery}, we say that a mean curvature flow $\mathcal{M}$ has a \emph{strong $\eps$-neck} with center $p$ and radius $r$ at time $t_0$, if the parabolically rescaled flow $\{r^{-1} (M_{t_0+r^2 t} -p)\}_{t\in(-1,0]}$ is $\eps$-close in $C^{\lfloor 1/\eps\rfloor}$ in $B_{1/\eps}(0) \times (-1, 0]$ to the evolution of a round cylinder $S^{n-1}\times \mathbb{R}$ (up to rotation) with center $0$ and radius $1$ at $t = 0$.

We say that $\mathcal{M}$ has a \emph{very strong $\eps$-neck} with center $p$ and radius $r$ at time $t_0$, if in the above definition $(-1,0]$ can be replaced by $(-2\mathcal{T},0]$, where $\mathcal{T}=\mathcal{T}(2\eps_1,\tfrac{1}{2}\eps_1,\alpha,\beta)$ denotes the constant from Proposition \ref{prop_backward_stab}.\footnote{Here, $\eps_1$ is a certain quality parameter, that will be fixed in Section \ref{sec_conclusion}.}

A subset $N\subset M_{\bar{t}}$ is called an \emph{$\eps$-tube}, if it is diffeomorphic to a cylinder and each $x\in N$ lies on the central sphere of a very strong $\eps$-neck at time $\bar{t}$ (more precisely, the slice of a strong $\eps$-neck at its maximal time).

Finally, as in \cite{BHH}, for each $\eps$-tube we can find an \emph{$\eps$-approximate central curve} $\gamma$. In particular, for each $p\in\gamma$ the vector $\partial_s\gamma(p)$ determines the axis of the $\eps$-neck centered at $p$ (recall that the axis of an $\eps$-neck is well defined up to errors of order $\eps$).

\subsection{Canonical neighborhoods}\label{sec_can_nbds}

We recall the canonical neighborhood theorem from Haslhofer-Kleiner \cite{HK_surgery} in the special case of smooth flows without surgeries.\footnote{In this special of smooth flows without surgeries the proof of the canonical neighborhood theorem simplifies quite a bit. In fact, Theorem \ref{thm_can_nbd} follows directly from the global convergence theorem for $\alpha$-noncollapsed flows \cite[Thm. 1.10]{HK}.} The general case of the canonical neighborhood theorem for mean curvature flow with surgery will only be needed at the end, and its discussion will thus be deferred until Section \ref{sec_sharp_integral}.

\begin{theorem}[{Special case of \cite[Thm 1.22]{HK_surgery}}]\label{thm_can_nbd}
For every $\eps_{\textrm{can}}>0$, there exist a constant $H_{\textrm{can}}(\eps_{\textrm{can}})=H_{\textrm{can}}(\Balpha,\eps_{\textrm{can}})<\infty$ with the following significance.
If $\mathcal{M}$ is a smooth mean curvature flow with $\Balpha$-controlled initial data, then any $(x,t)\in\mathcal{M}$ with $H(x,t)\geq H_{\textrm{can}}(\eps_{\textrm{can}})$ is $\eps_{\textrm{can}}$-close to a $\beta$-uniformly $2$-convex ancient $\alpha$-noncollapsed mean curvature flow.
\end{theorem}

For brevity we refer to ``$\beta$-uniformly $2$-convex ancient $\alpha$-noncollapsed mean curvature flows" simply as \emph{$(\alpha,\beta)$-solutions}. We recall that being $\eps_{\textrm{can}}$-close is understood in a scale invariant sense. The conclusion in Theorem \ref{thm_can_nbd} thus means that the flow $\mathcal{M}'$ which is obtained from $\mathcal{M}$ by shifting $(x,t)$ to the origin and parabolically rescaling by $H^{-1}(x,t)$ is $\eps_{\textrm{can}}$-close in $C^{\lfloor 1/\eps_{\textrm{can}}\rfloor}$ in $B_{1/\eps_{\textrm{can}}}(0)\times (-\eps_{\textrm{can}}^{-2},0]$ to an $(\alpha,\beta)$-solution. 

The structure of $(\alpha,\beta)$-solutions has been analyzed in \cite[Sec. 3.1]{HK_surgery}. In particular, $(\alpha,\beta)$-solutions are always convex and look neck-like away from caps of controlled size (in a scale invariant sense, as always).

\subsection{{\L}ojasiewicz-Simon inequality}\label{subsec_LS} For a hypersurface $M \subset \mathbb{R}^{n+1}$ the Gaussian surface area is defined by
\begin{equation}\label{huisken_density}
F(M)=(4\pi)^{-n/2}\int_M e^{-|x|^2/4} d\mathcal{H}^n.
\end{equation}
The entropy $\lambda$ is defined as supremum of the Gaussian area over all centers and scales:
\begin{equation}
\lambda(M)=\sup_{a>0,b\in\mathbb{R}^{n+1}}F(a M-b).
\end{equation}
By Huisken's monotonicity formula \cite{Huisken_monotonicity} the entropy is nonincreasing under mean curvature flow. Thus, the entropy of all hypersurfaces under consideration will be bounded above by some constant $\Lambda=\Lambda(\alpha,\gamma,\mathcal{A})$.
Also recall that if $\{M_t\subset\mathbb{R}^{n+1}\}_{t<0}$ moves by mean curvature flow, then $\Sigma_s=\tfrac{1}{\sqrt{-t}}M_t$, $s=-\log(-t)$, moves by rescaled mean curvature flow
\begin{equation}
\partial_s x= H+\tfrac{1}{2}x^\perp.
\end{equation}
The rescaled mean curvature flow is the negative gradient flow of the $F$-functional.

\begin{theorem}[{Colding-Minicozzi \cite[Thm. 6.1]{CM}}]\label{thm_lojsim}

There exist constants $K,\bar{R}<\infty$, $\eps_{L}>0$, $\mu<1$ (depending only on the dimension and an upper bound for the entropy) such that if $\{\Sigma_s\}_{s\in [t-1,t+1]}$ is a rescaled mean curvature flow such that $B_{\bar{R}}\cap M_s$ is for each $s$ a $C^{2,\alpha}$ graph with norm at most $\eps_{L}$ over the cylinder $Z=S^{n-1}_{\sqrt{2(n-1)}}\times \mathbb{R}$, then
\begin{equation}\label{discrete_LS}
|F(\Sigma_t)-F(Z)|^{1+\mu} \leq K \left( F(\Sigma_{t-1})-F(\Sigma_{t+1})\right).
\end{equation} 
\end{theorem}
To see that \eqref{discrete_LS} is indeed a discrete {\L}ojasiewicz-Simon gradient inequality, it helps to rewrite the right hand side using
\begin{equation}
F(\Sigma_{t-1})-F(\Sigma_{t+1})=\int_{t-1}^{t+1}|\nabla_{\Sigma_s} F|^2 \, ds.
\end{equation}

\section{Reduction to the cylindrical case}\label{sec_reduction}

Let $\mathcal{M}=\{M_t\}_{t\in [0,T)}$ be a 2-convex flow with $\Balpha$-controlled initial condition, and area bounded by $\mathcal{A}$. 

Let $\bar{t}<T$. We want to show that $\textrm{diam}(M_{\bar{t}})\leq C$, for some $C<\infty$ depending only on $\Balpha$ and $\mathcal{A}$.

\subsection{Reduction to the high curvature case}

We will argue first that it is enough to estimate the length of geodesics that stay entirely in regions of high curvature. To this end, assume that $\bar{H}$ is a constant with $\bar{H}\geq H_{\textrm{can}}(\eps_{\textrm{can}})$, where the precision parameters $\eps>0$ and $\eps_{\textrm{can}}=\eps_{\textrm{can}}(\eps,\Balpha)>0$ are small enough as in \cite{HK_surgery}, and consider the quantity
\begin{multline}
D(M_{\bar{t}}):=\sup \{ \ell(\gamma) \, | \, \textrm{$\gamma$ is a minimizing geodesic in $(M_{\bar{t}},d_{\bar{t}})$,}\\
\textrm{and $H>2\bar{H}$ along $\gamma$}\}.
\end{multline}

\begin{proposition}\label{prop_red1}
There exists a constant $\bar{N}=\bar{N}(\alpha,\mathcal{A},\bar{H})<\infty$ such that
\begin{equation}
\textrm{diam}(M_{\bar{t}})\leq \bar{N}+ (\bar{N}+1)D(M_{\bar{t}}).
\end{equation}
\end{proposition}

\begin{proof}
Let $\gamma:[0,L]\to (M_{\bar{t}},d_{\bar{t}})$ be a minimizing geodesic parametrized by arclength.
Choose a maximal collection $s_1,\ldots,s_N\in [0,L]$ such that $H(\gamma(s_i))\leq 2\bar{H}$ and $|s_i-s_j|\geq 1$.
Consider then the geodesic balls $B_i$ with center $\gamma(s_i)$ and radius $1/2$.

We claim that the curvature is uniformly bounded in small balls centered at $\gamma(s_i)$ of definite size (depending on $\bar H$ and $\Balpha$):
To this end, first observe that the $\alpha$-noncollapsing and the bound $H\leq \gamma$ at the initial time imply that $|A|\leq \sqrt{n}\gamma/\alpha$ at $t=0$. By standard doubling estimates there is a constant $T_1=T_1(\Balpha)>0$ such that if $\bar{t}\leq T_1$ then $|A|\leq 2\sqrt{n}\gamma/\alpha$ on the whole hypersurface. If $\bar{t}> T_1$ using that $H(\gamma(s_i))\leq 2\bar{H}$ we can apply the local curvature estimate \cite[Thm. 1.8]{HK} to get curvature control in a ball centered at $\gamma(s_i)$ of definite size.

By the above and the $\alpha$-noncollapsing we get a lower bound
\begin{equation}
\mathcal{H}^n(B_i)\geq c(\Balpha,\bar{H})>0.
\end{equation}
Since the balls $B_i$ are disjoint, this implies that
\begin{equation}
N\leq \mathcal{A}/c=:\bar{N}.
\end{equation}
Now restricting $\gamma$ to the connected components of $[0,L]\setminus \bigcup_{i=1}^N (s_i-\tfrac{1}{2},s_i+\tfrac{1}{2})$ gives at most $\bar{N}+1$ minimizing geodesics with the additional property that $H>2\bar{H}$ along them. The assertion follows.
\end{proof}

\subsection{Applying the canonical neighborhood theorem}

Consider the quantity
\begin{multline}
L(M_{\bar{t}}):=\sup\{ \textrm{diam}(N,d_{\bar{t}})\, | \, \textrm{$N\subset M_{\bar{t}}$ is an $\eps$-tube}\\
\textrm{with $H>\bar{H}$ on $N$} \}.
\end{multline}

\begin{proposition}\label{prop_red2}
There exists a constant $C=C(\Balpha,\eps,H_{\textrm{can}})<\infty$ such that
\begin{equation}
D(M_{\bar{t}})\leq C+C L(M_{\bar{t}}).
\end{equation}
\end{proposition}

\begin{proof}
Let $\gamma$ be a minimizing geodesic in $(M,d_{\bar{t}})$ with $H>2\bar{H}$ along $\gamma$. We can apply the canonical neighborhood theorem (Theorem \ref{thm_can_nbd}) at each point $(\gamma(s),\bar{t})$. If the canonical model at some point is a convex solution of controlled geometry, then $\gamma$ has controlled length and we are done. In all other cases, arguing as in the proof of \cite[Cor. 1.25]{HK_surgery} we see that $\gamma$ must be contained in an $\eps$-tube, possibly capped at one or both ends, or with its ends identified. The caps have diameter at most $C(\eps)H_{\textrm{can}}^{-1}$, and thus contribute at most a controlled amount to the length of $\gamma$. Similarly, in the case that the ends of the $\eps$-tube are identified, we can also simply ignore a piece of small size, to reduce it again to the case of an $\eps$-tube. The assertion follows.
\end{proof}

\section{Estimating the length of cylindrical regions}\label{cylindrical_regions}

\subsection{Small axis tilting under a priori assumptions}
In this section, we consider mean curvature flows under the a priori assumption that they are close to cylinders along a large -- possibly uncontrolled -- number of scales.
We assume that the space-time center point is fixed, but a priori the axis is allowed to be different at every scale. Using methods from \cite{CM}, we show that the total tilt of the axis is in fact small.

\begin{proposition}\label{prop_small_tilt}
For all $\eps_0>0$ there exists an $\eps_1=\eps_1(\eps_0,\Lambda)>0$ with the following significance. Let $\mathcal{M}$ be a mean curvature flow with entropy bounded by $\Lambda$, and suppose there are $p\in\mathbb{R}^{n+1}$ and $t_0<t_1<t_\ast$, such that for all $t\in[t_0,t_1]$, $\mathcal{M}$ has a strong $\eps_1$-neck with center $p$ and radius $\sqrt{2(n-1)(t_\ast-t)}$ at time $t$.
Then there exists a $v\in\mathbb{R}^{n+1}$ such that for all $t\in[t_0,t_1]$, $\mathcal{M}$ has a strong $\eps_0$-neck with center $p$, radius $\sqrt{2(n-1)(t_\ast-t)}$ at time $t$, and axis in the fixed direction $v$.
\end{proposition}

\begin{proof}
Our argument is a variant of the one from Colding-Minicozzi \cite[Sec. 6]{CM}. Consider the rescaled mean curvature flow $\Sigma_s=\frac{1}{\sqrt{t_\ast-t}}M_t$, $s=-\log(t_\ast-t)$. If $\eps_1$ is small enough (depending on $\varepsilon_L=\eps_L(\Lambda), \bar R=\bar{R}(\Lambda)$) then the hypothesis of Theorem \ref{thm_lojsim} are satisfied. Hence,
\begin{equation}
|F(\Sigma_s) - F(Z)|^{1+\mu} \leq K(F(\Sigma_{s-1}) - F(\Sigma_{s+1})),
\end{equation}
for every $s\in[s_0,s_1]$, where $s_0=\lceil - \log(t_\ast-t_0) +1\rceil, s_1=\lfloor -\log(t_\ast-t_1)-1\rfloor$.

Applying the discrete {\L}ojasiewicz lemma (Lemma \ref{sum}) for the function $f(s)=F(\Sigma_s)-F(Z)$ we infer that for every ${\eps}>0$ there exists an $\bar{\eps}=\bar{\eps}({\eps},\Lambda)>0$ such that
\begin{equation}
\sum_{j=s_0+1}^{s_1} (F(\Sigma_j) - F(\Sigma_{j-1}))^{\frac{1}{2}} <{\eps},
\end{equation}
provided $\eps_1<\bar{\eps}$. Using the Cauchy-Schwarz inequality and the fact that the rescaled mean curvature flow is the negative gradient flow of $F$, this implies
\begin{align}
&\int_{s_0}^{s_1 } \int_{\Sigma_s}  | \vec{H}+\tfrac{1}{2}x^\perp | \frac{e^{-|x|^2/4}}{(4\pi)^{n/2}} d\mathcal{H}^n\, ds \\
&\qquad\leq \Lambda^{1/2} \sum_{j=s_0+1}^{s_1} \left( \int_{j-1}^{j} \int_{\Sigma_s}  | \vec{H}+\tfrac{1}{2}x^\perp |^2 \frac{e^{-|x|^2/4}}{(4\pi)^{n/2}} d\mathcal{H}^n\, ds \right)^{1/2} \\
&\qquad= \Lambda^{1/2} \sum_{j=s_0+1}^{s_1} (F(\Sigma_j) - F(\Sigma_{j+1}))^{\frac{1}{2}} <\Lambda^{1/2}{\eps}.
\end{align}
Choosing $\eps_1=\eps_1(\eps_0,\Lambda)$ small enough, we can make the time-integral of the weighted $L^1$-norm of $\vec{H}+\tfrac{1}{2}x^\perp$ as small as we want, and the assertion follows.
\end{proof}

\subsection{Backwards stability of necks}

Inspired by \cite[Thm 6.1]{KL_singular}, we prove the following neck-stability result for $(\alpha,\beta)$-solutions:

\begin{proposition}\label{prop_backward_stab}
There is a $\delta_{\textrm{neck}}>0$ such that for all $\delta_0,\delta_1\leq \delta_{\textrm{neck}}$ there is a $\mathcal{T}=\mathcal{T}(\delta_0,\delta_1,\alpha,\beta)<\infty$ with the following property.\\
Suppose that $\mathcal{M}$ is an $(\alpha,\beta)$-solution that has a $\delta_0$-neck with center $p$ and radius $\sqrt{2(n-1)}$ at time $-1$. Then for all $t\in (-\infty, -\mathcal{T}]$ the flow $\hat{\mathcal{M}}$ which is obtained from $\mathcal{M}$ by shifting $(p,0)$ to the origin and parabolically rescaling by $|t|^{-1/2}$ is $\delta_1$-close to a round shrinking cylinder $S^{n-1}\times\mathbb{R}$ (up to rotation) that becomes extinct at time $0$. In particular, for all $t\in (-\infty, -\mathcal{T}]$ the unrescaled flow $\mathcal{M}$ has a strong $\delta_1$-neck with center $p$ and radius $\sqrt{2(n-1)|t|}$ at time $t$.
\end{proposition}

\begin{proof}
We first fix small enough constants $\delta_{\textrm{neck}}>0$ and $c>0$ (depending only on the dimension) such that the Gaussian density \eqref{huisken_density} satisfies
\begin{equation}\label{lowerdensity}
F(M_{-1})\geq \lambda(S^n)+c,
\end{equation}
whenever $\mathcal{M}$ has a $\delta_0$-neck with center $0$ and radius $\sqrt{2(n-1)}$ at time $-1$.

If the conclusion of the theorem didn't hold, then we could find a sequence $\mathcal{M}^i$ of $(\alpha,\beta)$-solutions with a $\delta_0$-neck with center $0$ and radius $\sqrt{2(n-1)}$ at time $-1$, but such that the flows $\hat{\mathcal{M}}^i$ which are obtained from $\mathcal{M}^i$ by parabolically rescaling by $|t_i|^{-1/2}$ are not $\delta_1$-close to a round shrinking cylinder, for some sequence $t_i\to -\infty$.

Consider Huisken's monotone quantity based at $(0,0)$, namely
\begin{equation}
\Phi (M_t^i)=\int_{M_t^i}\frac{1}{(4\pi\abs{t})^{n/2}}e^{-|x|^2/4|t|}d\mathcal{H}^n\qquad (t<0).
\end{equation}
Since for each fixed $i$, the flow $\hat{\mathcal{M}}^i$ has a blowdown-limit (ancient soliton) which must be either a plane, a round shrinking sphere, or a round shrinking cylinder with only one $\mathbb{R}$-factor (see \cite[Thm. 1.14, and its refinement for uniformly 2-convex flows]{HK}), we get the upper bound
\begin{equation}
\Phi (M_t^i)\leq \lambda(S^{n-1}).
\end{equation}

After passing to a subsequence, we can find $\tilde{t_i}\in (t_i,-1)$ with ${\tilde{t_i}}/t_i\to 0$ such that 
\begin{equation}
|\Phi(M^i_t)-\Phi(M^i_{\tilde t_i})|< 1/i
\end{equation}
for $t\in[A_i \tilde t_i,A_i^{-1}\tilde t_i]$, where $A_i\to\infty$. 

Let $\tilde{\mathcal{M}}^{i}$ be the sequence of flows that is obtained from $\mathcal{M}^i$ by parabolically rescaling by $|\tilde{t}_i|^{-1/2}$. After passing to a subsequence we can pass to a limit $\tilde{\mathcal{M}}^{i}\to \tilde{\mathcal{M}}^\infty$ which must be must be either a flat plane, a round shrinking sphere, or a round shrinking cylinder with only one $\mathbb{R}$-factor (see again \cite[Thm. 1.14, and its refinement for uniformly 2-convex flows]{HK}). The plane and the sphere are excluded by \eqref{lowerdensity}. In particular, we see that $\Phi (M^i_{\tilde{t}_i})\to \lambda(S^{n-1})$. 

Thus, for $i$ large enough $\Phi(M^i_t)$ is almost constant on the interval $(-\infty,\tilde{t}_i)$. Consequently, after passing to another subsequence we can pass to a limit $\hat{\mathcal{M}}^i\to \hat{\mathcal{M}}^\infty$ which must be a round shrinking cylinder with one $\mathbb{R}$-factor (see again \cite[Thm. 1.14, and its refinement for uniformly 2-convex flows]{HK}); this gives the desired contradiction.
\end{proof}

\subsection{Conclusion of the argument}\label{sec_conclusion}

Let $\mathcal{M}=\{M_t\}_{t\in [0,T)}$ be a 2-convex flow with $\Balpha$-controlled initial condition, and area bounded by $\mathcal{A}$. 
Let $\bar{t}<T$. We want to show that $\textrm{diam}(M_{\bar{t}})\leq C$, for some constant $C=C(\Balpha,\mathcal{A})<\infty$.

The following argument depends on various quality parameters for necks. The logical order for choosing these parameters is that one first fixes a small constant $\eps_0$ (depending only on the dimension), and a large factor $Q<\infty$ (e.g. Q=10), and then successively determines suitable quality parameters $\eps_1$, $\eps$, and $\eps_{\textrm{can}}$ (where $\eps_{\textrm{can}}\ll \eps\ll \eps_1 \ll \eps_0$) by reading the argument backwards.

Let $N\subset M_{\bar{t}}$ by an $\eps$-tube with $H>\bar{H}=  Q H_{\textrm{can}}$ on $N$. By the reduction from Section \ref{sec_reduction}, namely by Proposition \ref{prop_red1} and Proposition \ref{prop_red2}, it is enough to estimate the length of $N$.

Let $\gamma:[0,L]\to\mathbb{R}^{n+1}$ be an $\eps$-approximate central curve for $N$ (see Section \ref{sec_neck_tubes}), parametrized by arclength. We want to establish an upper bound for $L$, depending only on $\Balpha$ and $\mathcal{A}$. Suppose $L>2$ (otherwise we are done), and consider the truncated curve $\gamma|_{[1,L-1]}$ (which we still denote by $\gamma$) and the corresponding truncated $\eps$-tube around it (which we still denote by $N$).

Given any $p\in\gamma$, let $x\in N$ be a point on the central sphere of the $\eps$-neck centered at $p$ of radius $r_p$ at time $\bar t$. Note that $|x-p|\leq(1+\eps)r_p$. Also note that,
since $H(x)>\bar H$, for $\eps$ small enough we have 
\begin{equation}\label{eqn_rsmall}
r_p\leq 2(n-1) (Q H_{\textrm{can}})^{-1}.
\end{equation}

Set $t_p=\bar t+\tfrac{1}{2(n-1)} r_p^2$, and let $\tau:=\tfrac{1}{16}H_{\textrm{can}}^{-2}$. We claim that for every $p\in\gamma$ and every $t\in [\bar{t}-\tau,\bar{t}]$, the flow $\mathcal M$ has a strong $\eps_1$-neck with center $p$ and radius 
\begin{equation}\label{eqn_rformula}
r(t)=\sqrt{2(n-1)(t_p-t)},
\end{equation}
at time $t$, provided $\eps$ and $\eps_{\textrm{can}}$ are small enough. Note that the axis is a priori allowed to change when going from scale to scale.

Suppose towards a contradiction that $t_0\in [\bar t-\tau,\bar t]$ is the largest time such that $\mathcal M$ does not have a strong $\eps_1$-neck with center $p$ and radius $r(t_0)$ at time $t_0$.

Since every point in an $\eps$-tube belongs to the final time slice of a very strong $\eps$-neck, we immediately see that $t_0<\bar t$. More precisely, given $\eps_1$ let $\mathcal T= \mathcal T(2\eps_1,\tfrac{1}{2}\eps_1,\alpha,\beta)$ be the constant from Proposition \ref{prop_backward_stab} (backwards stability). Since there is a very strong $\eps$-neck centered at $p$ of radius $r(\bar t)$ at time $\bar t$, then for every $t$ satisfying 
\begin{equation}\label{verystrong_strong}
t-r(t)^2\geq\bar t-2\mathcal T r(\bar t)^2,
\end{equation} 
there is a strong $\eps_1$-neck centered at $p$ of radius $r(t)$ at time $t$. At time $t=t_0$ the inequality \eqref{verystrong_strong} must be violated, in other words
\begin{equation}\label{t0_is_far}
\frac{t_p - t_0}{t_p-\bar t}> \frac{4(n-1)\mathcal T+1}{2(n-1)+1}\geq \frac32 \mathcal T,
\end{equation}
where we tacitly assume that $\mathcal T$ is large enough (depending only on the dimension). Since obviously $\frac{t_p - t_0}{t_p-t_0}=1$, by the intermediate value theorem we can find a $t_1\in (t_0,\bar t]$ such that 
\begin{equation}\label{time_scaling}
\frac{t_p- t_0}{t_p-t_1}=\frac32 \mathcal T.
\end{equation}

Since $t_1\in (t_0,\bar t]$, the flow $\mathcal M$ has an $\eps_1$-neck centered at $p$ of radius  $r(t_1)$ at time $t=t_1$. Hence, there is a point $x_{t_1}\in M_{t_1}$ on the central sphere of that neck, which satisfies in particular
\begin{equation}\label{neck_curvature}
\tfrac{n-1}{(1+\eps_1)r(t_1)}<H(x_{t_1})<\tfrac{(1+\eps_1)(n-1)}{r(t_1)},
\end{equation}
and
\begin{equation}\label{distance_center}
|x_{t_1}-p| < (1+\eps_1)r(t_1).
\end{equation}

Using \eqref{eqn_rsmall}, \eqref{eqn_rformula}, and our choice of $Q$ and $\tau$ we see that
\begin{equation}
r^2(\bar t-\tau)= r_p^2+2(n-1)\tau\leq \frac{(n-1)^2}{4}H_{\textrm{can}}^{-2},
\end{equation}
which together with \eqref{neck_curvature} and the obvious inequality $r(t_1)\leq r(\bar t-\tau)$, implies that
\begin{equation}
H(x_{t_1})\geq H_{\textrm{can}}.
\end{equation}
Thus, by Theorem \ref{thm_can_nbd} (canonical neighborhoods) and \eqref{distance_center} the flow $\hat{ \mathcal M}$ that is obtained from $\mathcal M$ by shifting $(p,t_1)$ to the origin and parabolically rescaling by $H^{-1}(x_{t_1})$ is $\eps_{\textrm{can}}$-close in $C^{\lfloor 1/ \eps_{\textrm{can}}\rfloor}$ in $B_{1/\eps_{\textrm{can}}-2n}(0)\times (-\eps_{\textrm{can}}^2,0]$ to an $(\alpha,\beta)$-solution $\mathcal{N}$.

For $\eps_{\textrm{can}}$ small enough, the $(\alpha,\beta)$-solution $\mathcal N$ has a $2\eps_1$-neck with radius $(n-1)$ and center $0$ at time $0$. By  Proposition \ref{prop_backward_stab} (backwards stability) it follows that $\mathcal N$ has a strong $\tfrac{1}{2}\eps_1$-neck centered at $0$ of radius $\sqrt{2(n-1)(s+(n-1)/2) }$ at time $-s$, provided $s\geq (n-1)(\mathcal{T}-1)/2$. 

Putting things together, it follows that if $\eps_{\textrm{can}}\ll \eps_1$ (depending only on $\Balpha$) then $\mathcal M$ has a strong $\tfrac{3}{4}\eps_1$-neck centered at $p$ of radius $r(t_0)=\sqrt{2(n-1)(t_p-t_0)}$ at time $t_0$; a contradiction. This proves the claim.

Now we can apply Proposition \ref{prop_small_tilt} (small tilt) with $\eps_1=\eps_1(\eps_0,\Lambda)$ to conclude that for every $p\in\gamma$ there exists $O_p\in \textrm{SO}_{n+1}$ such that for all $t\in [\bar{t}-\tau,\bar{t}]$ we have that $M_t$ is $\eps_0$-close to the cylinder $Z_p=p+O_p(S^{n-1}_{\sqrt{2(n-1)(t_p-t)}}\times \mathbb{R})$, with fixed axis, in $B_{\eps_0^{-1}\sqrt{2(n-1)(t_p-t)}}(p)$.

Finally, if $p_1,p_2\in \gamma$ are points with $|p_1-p_2|<\tfrac{1}{4}\eps_0^{-1}\sqrt{\tau}$, then the associated cylinders $Z_{p_1},Z_{p_2}$ have substantial overlap at time $t=\bar{t}-\tau$, hence $||O_{p_1} -O_{p_2}||=O(\varepsilon_0)$ and $\min\left\{\frac{t_{p_1}}{t_{p_2}}, \frac{t_{p_2}}{t_{p_1}}\right\}=1+O(\varepsilon_0)$. Thus
\begin{equation}\label{intrinsic_extrinsic}
d_\gamma(p_1,p_2)\leq (1+O(\eps_0))|p_1-p_2|,
\end{equation}
i.e. the intrinsic distance along $\gamma$ between any two points $p_1,p_2\in \gamma$ with $|p_1-p_2|<\tfrac{1}{4}\eps_0^{-1}\sqrt{\tau}$ is controlled by $(1+O(\eps_0))$ times their extrinsic distance. Thus, the intersection of $\gamma$ with any ball of radius $\tfrac{1}{8}\eps_0^{-1}\sqrt{\tau}$ is $O(\eps_0)$-close to a linear segment. Since $M_0$ is contained in  large ball of radius $R=R(\Balpha,\mathcal{A})<\infty$, we conclude that the length of $\gamma$ is bounded by some constant depending only on $\Balpha$ and $\mathcal{A}$.
This finishes the proof of Theorem \ref{thm_diam}.

\section{Proof of Theorem \ref{thm_regularity_est}}\label{sec_sharp_integral}
 
In this final section, we will prove Theorem \ref{thm_regularity_est}. 

We use the framework of mean curvature flow with surgery from \cite{HK_surgery}. Thus, for us a mean curvature flow with surgery is a $(\Balpha,\delta,\mathbb{H})$-flow as defined in \cite[Def. 1.17]{HK_surgery}. Recall in particular that $\Balpha=(\alpha,\beta,\gamma)$ denotes the control parameters for the two-convex initial hypersurface $M_0$, that $\delta$ specifies the quality of the surgeries, and that the three curvature scales $\mathbb{H}=(H_{\textrm{trig}},H_{\textrm{neck}},H_{\textrm{th}})$ are used to specify more precisely when and how surgeries (and/or discarding) are performed.

The main existence theorem \cite[Thm. 1.21]{HK_surgery} and the canonical neighborhood theorem \cite[Thm. 1.22]{HK_surgery}, tell us that for any $\Balpha$-controlled initial hypersurface $M_0$, there exists an $(\Balpha,\delta,\mathbb{H})$-flow starting at $M_0$ such that all points with $H\geq H_{\textrm{can}}(\Balpha, \eps_{\textrm{can}})$ possess canonical neighborhoods, provided $\delta$ is small enough, and $H_{\textrm{th}}$ as well as the ratios $H_{\textrm{trig}}/H_{\textrm{neck}}$, $H_{\textrm{neck}}/H_{\textrm{th}}$ are large enough (depending only on $\Balpha$).

We make the following two minor adjustments compared to \cite{HK_surgery}. First, we work with `very strong' necks instead of `strong' necks. This is just a cosmetic change, that only slightly alters some constants. Second, and more importantly, in the line before \cite[Claim 4.7]{HK_surgery} instead of selecting an arbitrary minimal collection of disjoint $\delta$-necks that separates the thick part and the trigger part, we select an `innermost' collection of such $\delta$-necks, i.e. we impose the additional condition that
\begin{equation}
\sum_{p\in \hat{\mathcal{J}}_j} \textrm{dist}(p,\{H=H_{\textrm{trig}}\})
\end{equation}
is minimal. If the surgeries are performed this way, then on the discarded components we have
\begin{equation}\label{eq_lowerdisc}
H \geq c H_{\textrm{neck}}
\end{equation}
for some $c=c(\Balpha)>0$.

The argument from the previous sections (taking into account the one additional case that canonical neighborhoods can now also be modelled on the evolution of a standard cap preceeded by the evolution of a round shrinking cylinder) shows that every connected component $M_{\bar{t}}'\subset M_{\bar{t}}$ satisfies
\begin{equation}\label{diam_surgery}
\textrm{diam}(M_{\bar{t}}',d_{\bar{t}})\leq C(\Balpha,\mathcal{A})<\infty
\end{equation}
for all $\bar{t}\geq 0$, where $d_{\bar t}$ denotes the intrinsic distance on $M'_{\bar t}$.

Moreover, by the canonical neighborhood theorem, the nature of the surgeries, and the $\alpha$-noncollapsing, the number of connected components is uniformly bounded by some $N=N(\Balpha,\mathcal{A},H_{\textrm{thick}})<\infty$.

We want to show that for any $\bar{t}\geq 0$ we have the estimate
\begin{equation}\label{lpmc}
\int_{M_{\bar t}'} H^{n-1} d\mu_{\bar{t}} \leq C(\Balpha,\mathcal{A})<\infty.
\end{equation}

To this end, write $M'_{\bar t}=M_{\bar t}^{\textrm{low}}\cup M_{\bar t}^{\textrm{high}}$, with $M_{\bar t}^{\textrm{low}}=\{x\in M'_{\bar{t}}: H(x)\leq \bar H\}$ and $M_{\bar t}^{\textrm{high}}=\{x\in M'_{\bar{t}}: H(x)> \bar H\}$, where $\bar{H}\gg H_{\textrm{can}}$ is a large but fixed constant as in the previous section.

Using the canonical neighborhood theorem, we can decompose $M_{\bar t}^{\textrm{high}}$ into the union of $\eps$-tubes, caps, standard caps, and compact solutions of controlled geometry. It is easy to see that the latter three only contribute a controlled amount to the integral in \eqref{lpmc}. 

Hence, let $M^{\textrm{tubes}}_{\bar t}\subset M_{\bar t}^{\textrm{high}}$ be the union of the remaining $\eps$-tubes, with curvature larger than $\bar H$. Note that at this point we allow the $\eps$-tubes to have identified ends and combine tubes with their ends attached. 

We can then separate the connected components of $M_{\bar t}^{\textrm{tubes}}$ in two types: $\eps$-tubes $\{T^{\textrm{low}}_i\}$ with curvature less than $2\bar H$ and $\eps$-tubes $\{T^{\textrm{high}}_i\}$ that intersect the set $\{H>2\bar H\}$. Since we obviously have
\begin{equation}
\int_{M_{\bar t}\cap \{H\leq 2\bar{H}\}} H^{n-1} d\mu_{\bar{t}} \leq    (2\bar H)^{n-1} \mathcal A,
\end{equation}
it remains to estimate
\begin{equation}\label{integral_high}
\int_{\bigcup_i T^{\textrm{high}}_i} H^{n-1} d\mu_{\bar{t}}.
\end{equation}

Now, we can write any $\eps$-tube $T^{\textrm{high}}$ as the union of a maximal collection of $\eps$-necks $N_i$ centered at $p_i$ of radius $r_i$ at time $t=\bar t$, such that any two balls $B_{r_i/5}(p_i)$, $B_{r_j/5}(p_j)$ are disjoint. We can then estimate
\begin{equation}\label{T_high}
\int_{T^{\textrm{high}}} H^{n-1}d\mu_t\leq \sum_i \int_{N_i} H^{n-1}d\mu_t \leq c \sum_i r_i \leq c L,
\end{equation}
where $L$ is the length of the $\eps$-tube $N$ and $c=c(n)<\infty$ is a numerical constant. Since the intrinsic diameter of each connected component is uniformly bounded, the integral \eqref{T_high} is bounded by some constant $C=C(\Balpha,\mathcal{A})<\infty$.

To bound the integral \eqref{integral_high}, we need to control the number of tubes in the union. Assume without essential loss of generality that $M_{\bar t}^{\textrm{low}}\neq \emptyset$ (the  case $M_{\bar t}^{\textrm{low}}= \emptyset$ can be handled easily). Note that each tube $T^{\textrm{high}}$ has a point $x$ with $H(x)=\tfrac{3}{2} \bar H$. This implies that $x$ belongs to an $\eps$-neck that is entirely contained in $T^{\textrm{high}}$ since its curvature is approximately $\tfrac{3}{2}\bar H$, so  each tube $T^{\textrm{high}}_i$ contributes a definite amount of area. Since the total area of $M_{\bar t}$ is bounded by $\mathcal A$, this suffices to control the number of such tubes by a constant that depends only on $\Balpha$ and $\mathcal A$, and concludes the proof of \eqref{lpmc}.

Now, in order to prove the regularity estimate for the level set flow, fix the initial hypersurface $M_0$, and consider a sequence $(\Balpha,\delta,\mathbb{H}^j)$-flows starting at $M_0$, with $H_{\textrm{neck}}^j\to \infty$, but $H^j_{\textrm{thick}}$ bounded. By the above, we have the uniform estimate
\begin{equation}
\int_{M_t^j} H^{n-1}d\mu_t \leq C(\Balpha,\mathcal{A}).
\end{equation} 
By the work of Head \cite{Head} and Lauer \cite{Lauer} (see also \cite[Prop. 1.27, Rem. 4.11]{HK}; in particular, note that equation \eqref{eq_lowerdisc} ensures that the curvature of the discarded components goes to infinity) for $j\to \infty$ the flows with surgery $\{M_t^j\}$ converge to the level set flow $\{M_t\}$ with initial condition $M_0$. The convergence is both in the space-time Hausdorff sense and also in the varifold sense for every time. By lower-semicontinuity of $L^p$-norms of the mean curvature  under varifold convergence, we thus infer that the level set flow satisfies
\begin{equation}
\int_{M_t} H^{n-1}d\mu_t \leq C(\Balpha,\mathcal{A})
\end{equation} 
for every $t\geq 0$. Finally, by the local curvature estimate \cite[Thm. 1.8]{HK} we have $r_{\mathcal{M}}^{-1}\leq C(\Balpha)H$, and the assertion of Theorem \ref{thm_regularity_est} follows.

\begin{appendix}

\section{Discrete {\L}ojasiewicz lemma}

\begin{lemma}[{Discrete {\L}ojasiewicz lemma, c.f. \cite[Lem. 6.9]{CM}}]\label{sum}
For every $\eps>0$, there exists a $\delta=\delta(\eps,K,\mu)>0$ with the following significance.
Suppose that $f:\{0,1,\ldots, T\}\rightarrow \mathbb{R}$ is a non-increasing function such that for some $K<\infty$ and $\mu<1$ it holds that
\begin{equation}
|f(t)|^{1+\mu} \leq K\left( f(t-1) - f(t+1)\right)
\end{equation}
for $t= 1,\ldots ,T-1$, and suppose that $|f|\leq \delta$. Then
\begin{equation}\label{full_sum}
\sum_{j=1}^{T}  (f(j) - f(j-1))^{\frac{1}{2}} \leq \eps.
\end{equation}
\end{lemma}

\begin{proof}
For most parts of the proof we only assume $|f|\leq 1$. We will impose the condition that $|f|\leq \delta$ is actually small towards the end.

Let $t_0\in \{0,1,\ldots,T\}$ be the smallest integer with the property $f(j)< 0$ for every $t_0<j \leq T$. 

If $t_0>0$, then as in \cite[Lemma 6.9]{CM}, there is a $C=C(K,\mu)<\infty$ such that
\begin{equation}\label{decay}
f(t)\leq C t^{-\frac{1}{\mu}},
\end{equation}
for every $t\in [0,t_0]$. Moreover, for $p\in (1,\frac{1}{\mu})$ and any $j_0\in [1,t_0]$, using the Cauchy--Schwarz inequality we obtain
\begin{equation}\label{CS}
\left(\sum_{j=j_0}^{t_0} (f(j) - f(j+1))^{\frac{1}{2}}\right)^{2} \leq \left(\sum_{j=j_0}^{t_0} (f(j) - f(j+1)) j^p\right) \sum_{j=j_0}^{t_0} j^{-p}.
\end{equation}
Estimating the right-hand side of \eqref{CS} as in \cite[Lemma 6.9]{CM} we find $b_1=b_1(\delta,K,\mu)<\infty$ such that if $b_1\leq j_0\leq t_0$ then
\begin{equation}\label{small1}
\sum_{j=j_0}^{t_0} (f(j) - f(j+1))^{\frac{1}{2}}<\frac{\varepsilon}{4}.
\end{equation}
In case that $t_0<T$, we also consider the function $\tilde f(t):=-f(T-t)$. This function $\tilde f:[0, T-t_0]\rightarrow [0,\infty)$ is non-increasing and satisfies
\begin{equation}
\begin{aligned}
\tilde f(t)^{1+\mu}=|f(T-t)|^{1+\mu}&\leq K\left( f(T-t-1)- f(T-t+1) \right)\\
&=K\left( \tilde f(t-1) -\tilde f(t+1)\right),
\end{aligned}
\end{equation}
and $|\tilde f|\leq 1$. As above, we can find $b_2=b_2(\delta,K,\mu)<\infty$ such that if $b_2\leq j_0\leq T-t_0-2$ then
\begin{equation}\label{small2}
\sum_{j=t_0+1}^{T-j_0-1} (f(j) - f(j+1))^{\frac{1}{2}}=\sum_{k=j_0}^{T-t_0-2} (\tilde f(k) - \tilde f(k+1))^{\frac{1}{2}}<\frac{\varepsilon}{4}.
\end{equation}
Together with \eqref{small1} this implies that for $b=\max\{b_1,b_2+1\}$
\begin{equation}\label{sum1}
\sum_{j=b}^{T-b} (f(j)-f(j+1))^{\frac{1}{2}}<\frac{\varepsilon}{2},
\end{equation}
tacitly assuming that $T>2b$ (otherwise there is not much to prove, see the next sentence).
Since only $2b-2$ terms of \eqref{full_sum} missing from \eqref{sum1}, and each is bounded by $(f(j)-f(j+1))^{\frac{1}{2}}<(2\delta)^{\frac{1}{2}}$, we conclude that \eqref{full_sum} holds, provided we choose $\delta$ small enough.
\end{proof}

\end{appendix}

\bibliographystyle{amsplain}

\newcommand{\noopsort}[1]{} \newcommand{\singleletter}[1]{#1}
\providecommand{\bysame}{\leavevmode\hbox to3em{\hrulefill}\thinspace}
\providecommand{\MR}{\relax\ifhmode\unskip\space\fi MR }
\providecommand{\MRhref}[2]{%
  \href{http://www.ams.org/mathscinet-getitem?mr=#1}{#2}
}
\providecommand{\href}[2]{#2}

\vspace{10mm}
{\sc Department of Mathematics, University of Toronto,  40 St George Street, Toronto, ON M5S 2E4, Canada}\\

\emph{E-mail:} p.gianniotis@utoronto.ca, roberth@math.toronto.edu

\end{document}